\documentclass[a4paper,12pt]{article}
\usepackage[latin1]{inputenc}
\usepackage{amsthm}
\usepackage{amsfonts}
\usepackage[dvips]{graphicx}
\usepackage{fancyhdr}
\usepackage{latexsym}
\usepackage{amsmath}

\newtheorem{thm}{Theorem}
\newtheorem{defn}{Definition}

\newtheorem{cor}{Corollary}
\newtheorem{oss}{Remark}

\allowdisplaybreaks[4]

\begin{document}

\title{\large\textbf{GLOBAL EXISTENCE OF WEAK SOLUTIONS
TO A \\ NONLOCAL CAHN-HILLIARD-NAVIER-STOKES SYSTEM}}

\author{
{\sc Pierluigi Colli}\\
Dipartimento di Matematica
\\Universit\`{a} degli Studi di Pavia\\Pavia I-27100, Italy\\
\textit{pierluigi.colli@unipv.it}
\\
\\
{\sc Sergio Frigeri}\\
Dipartimento di Matematica
\\Universit\`{a} degli Studi di Pavia\\Pavia I-27100, Italy\\
\textit{sergio.frigeri@unipv.it}
\\
\\
{\sc Maurizio Grasselli}\\
Dipartimento di Matematica\\
Politecnico di Milano\\
Milano I-20133, Italy \\
\textit{maurizio.grasselli@polimi.it}}

\date{}

\maketitle
\begin{abstract}
A well-known diffuse interface model consists of the Navier-Stokes
equations nonlinearly coupled with a convective Cahn-Hilliard type
equation. This system describes the evolution of an incompressible
isothermal mixture of binary fluids and it has been investigated by
many authors. Here we consider a variant of this model where the
standard Cahn-Hilliard equation is replaced by its nonlocal version.
More precisely, the gradient term in the free energy functional is
replaced by a spatial convolution operator acting on the order
parameter $\varphi$, while the potential $F$ may have any
polynomial growth. Therefore the coupling with the Navier-Stokes
equations is difficult to handle even in two spatial dimensions because
of the lack of regularity of $\varphi$. We establish the global
existence of a weak solution. In the two-dimensional case we also
prove that such a solution satisfies the energy identity and a
dissipative estimate, provided that $F$ fulfills a suitable coercivity
condition.
\end{abstract}

\section{Introduction}
\setcounter{equation}{0}
A well-known model which describes the evolution of an
incompressible isothermal mixture of two immiscible fluids is the
so-called model H (see \cite{HH,GPV}, cf. also \cite{D,LT,M} and
references therein). This is a diffuse-interface model (cf. \cite{AMW})
in which the sharp interface separating the two fluids (e.g., oil and
water) is replaced by a diffuse one by introducing an order parameter
$\varphi$. The dynamics of $\varphi$, which represents the (relative)
concentration of one of the fluids (or the difference of the two
concentrations), is governed by a Cahn-Hilliard type equation with a
transport term. This parameter influences the (average) fluid velocity
$u$ through a capillarity force (called Korteweg force) proportional to
$\mu\nabla\varphi$, where $\mu$ is the chemical potential. Note
that this force is concentrated close to the diffuse interface.

In a simplified setting where the density $\varrho$ of the mixture is
supposed to be constant as well as the viscosity $\nu$ and the
mobility $m$, the model reduces to
\begin{eqnarray}
& & \varphi_t+u\cdot\nabla\varphi=m\Delta\mu \label{CHNS1}\\
& &\varrho u_t-\nu\Delta u+(u\cdot\nabla)u+\nabla\pi=
\kappa\mu\nabla\varphi+h \label{CHNS2}\\
& &\mbox{div}(u)=0 \label{CHNS3}
\end{eqnarray}
in $\Omega\times(0,T)$, where $\Omega$ is a domain in
$\mathbb{R}^d$, $d=2,3$, $T>0$ is a given final time, $\pi$ is the
pressure, $\kappa$ is a given positive constant and $h$ represents
volume forces applied to the binary mixture fluid. The chemical
potential $\mu$ is the first variation of the free energy functional (see
\cite{CH})
\begin{equation}
\label{energyCH}
E(\varphi) = \int_\Omega \left(\frac{\xi}{2}\vert\nabla \varphi(x)\vert^2
+ \eta F(\varphi(x))\right)dx.
\end{equation}
Here $F$ represents the (density of) potential energy. This function is usually
a double-well potential whose wells are located in the pure
phases, while $\xi$ and $\eta$ are given positive constants. The
potential can be defined either on the whole real line (smooth
potential) or on a bounded interval (singular potential). The latter case
(in a logarithmic form) is the most appropriate choice from the
modeling viewpoint (cf. \cite{CH}), while the former can be considered
as an approximation.

In the context of statistical mechanics, the square gradient term in
\eqref{energyCH} arises from attractive long-ranged interactions
between the molecules of the fluid and $\xi$ can be related to the pair
correlation function (see, e.g., \cite{AMW} and references therein).
We also recall that $\kappa$ and $\xi$ are of the same order as the
interface thickness $\varepsilon>0$, while $\eta$ is proportional to
$\varepsilon^{-1}$. On account of \eqref{energyCH},  the chemical
potential takes the following form
\begin{equation}
\mu = -\xi\Delta \varphi + \eta F^\prime(\varphi). \label{chempot}
\end{equation}
Systems like \eqref{CHNS1}-\eqref{chempot}, also known as
Cahn-Hilliard-Navier-Stokes systems, have been studied from the
mathematical viewpoint by several authors  (see, for instance,
\cite{A1,A2,A3, B, B2,B3,GG1,GG2,S, ZWH}, cf. also \cite{F,KSW,SY}
for numerical issues).

A different form of the free energy has been proposed in
\cite{GL1,GL2} and rigorously justified as a macroscopic limit of
microscopic phase segregation models with particle conserving
dynamics (see also \cite{CF}). In this case the gradient term is
replaced by a nonlocal spatial operator, namely,
\begin{equation}
\label{energyNCH}
\mathcal{E}(\varphi) = \frac{1}{4}
\int_{\Omega}\int_{\Omega}J(x-y)
(\varphi(x)-\varphi(y))^2 dxdy + \eta\int_\Omega F(\varphi(x))dx,
\end{equation}
where $J:\mathbb{R}^d \to \mathbb{R}$ is a smooth function such
that $J(x)=J(-x)$. Taking the first variation of $\mathcal{E}$ we can
define the chemical potential associated with the nonlocal model
\begin{equation}
\label{mu}
\mu=a\varphi-J\ast\varphi+\eta F^\prime(\varphi)
\end{equation}
where
\begin{equation}
(J\ast\varphi)(x):=\int_{\Omega}J(x-y)\varphi(y)dy,\qquad a(x):=
\int_{\Omega}J(x-y)dy,\qquad x\in\Omega.
\label{defp}
\end{equation}
The corresponding nonlocal Cahn-Hilliard equation $\varphi_t
=m\Delta\mu$ can be derived from idealized microscopic models
through suitable limits like the diffusion equation and the Boltzmann
equation.  Moreover, the evolution in the sharp interface limits are the
same as those derived from the classical Cahn-Hilliard equation in the
corresponding limits (see \cite{GL2}). However, from the
mathematical viewpoint, the nonlocal Cahn-Hilliard equation, due to
its integrodifferential nature, is rather difficult to handle (see, e.g.,
\cite{BH1,BH2,CKRS,G,GZ,H,LP}). Here we consider system
\eqref{CHNS1}-\eqref{CHNS3} with \eqref{mu}. More precisely,
taking for simplicity all the constants but $\nu$ equal to one, we want
to study the following initial and boundary value problem
\begin{eqnarray}
& & \varphi_t+u\cdot\nabla\varphi=\Delta\mu\label{sy1}\\
& &\mu=a\varphi-J\ast\varphi+F^\prime(\varphi)\\
& &u_t-\nu\Delta u+(u\cdot\nabla)u+\nabla\pi=\mu\nabla\varphi+h\\
& &\mbox{div}(u)=0\\
& &\frac{\partial\mu}{\partial n}=0,\quad u=0\quad\mbox{on }
\partial\Omega\times (0,T)\label{nslip}\\
& &u(0)=u_0,\quad\varphi(0)=\varphi_0\quad\mbox{in }\Omega,
\label{sy6}
\end{eqnarray}
where $\Omega\subset\mathbb{R}^d$, $d=2,3$, is a bounded
domain with sufficiently smooth boundary and unit outward normal
$n$. The no-flux boundary condition for $\mu$ is the usual one for
Cahn-Hilliard type equations (cf., e.g., \cite{BH1}) and implies the
conservation of mass (see Remark \ref{mass} below). The no-slip
boundary condition for $u$ is also standard especially when one wants
to investigate a new model involving Navier-Stokes equations
(periodic boundary conditions can also be considered).

In this contribution we prove the existence of a global weak solution
for smooth potentials $F$ of arbitrary polynomial growth. Moreover, if
$F$ satisfies a suitable coercivity condition then we can slightly improve
the smoothness properties of the solution.
In particular, we show the validity of an energy identity if $d=2$.
These results are a first step towards the mathematical analysis of problem
\eqref{sy1}-\eqref{sy6}.  However, further issues (such as, e.g., uniqueness
in two dimensions) do not seem so straightforward to prove. The main
difficulty arises from the presence of the nonlocal term which implies
that $\varphi$ is not as regular as for the standard (local)
Cahn-Hilliard-Navier-Stokes system (cf. Remark \ref{smooth} below).
For this reason, we have not been able even to establish uniqueness of
weak solutions in two dimensions.

\section{Notation and functional setup}\setcounter{equation}{0}
Let us set $V_s:=D(B^{s/2})$ for every $s\in\mathbb{R}$, where
$B=-\Delta+I$ with homogeneous Neumann boundary conditions.
Hence we have
$$
V_2=D(B)=\left\{v\in H^2(\Omega):\;
\frac{\partial v}{\partial n} =0\mbox{
on }\partial\Omega\right\}.
$$
We also define $H:=V_0=L^2(\Omega)$ and
$V:=V_1=H^1(\Omega)$. Then we introduce the classical Hilbert
spaces for the Navier-Stokes equations (see, e.g., \cite{T})
$$G_{div}:=\overline{\{u\in C^\infty_0(\Omega)^d:\mbox{ div}(u)=0\}}^{L^2(\Omega)^d},$$
and
$$V_{div}:=\{u\in H_0^1(\Omega)^d:\mbox{ div}(u)=0\}.$$
We denote by $\|\cdot\|$ and $(\cdot,\cdot)$ the norm and the scalar product, respectively,
on both $H$ and $G_{div}$. We recall that $V_{div}$ is endowed with the scalar product
$$(u,v)_{V_{div}}=(\nabla u,\nabla v),\qquad\forall u,v\in V_{div}.$$
We also need to introduce the Stokes operator  $A:D(A)\cap
G_{div}\to G_{div}$. Recall that, in the case of no-slip boundary
condition \eqref{nslip}
$$A=-P\Delta,\qquad D(A)=H^2(\Omega)^d\cap V_{div},$$
where $P:L^2(\Omega)^d\to G_{div}$ is the Leray projector. Notice that we have
$$(Au,v)=(u,v)_{V_{div}}=(\nabla u,\nabla v),\qquad\forall u\in D(A),\quad\forall v\in V_{div}.$$
We also recall that $A^{-1}:G_{div}\to G_{div}$ is a self-adjoint compact operator in $G_{div}$
and by the classical spectral theorems there exists a sequence $\lambda_j$ with $$0<\lambda_1\leq\lambda_2\leq\cdots,\qquad\lambda_j\to\infty,$$
and a family of $w_j\in D(A)$ which is orthonormal in $G_{div}$ and such that
$$Aw_j=\lambda_jw_j.$$
Finally, for $u$, $v$, $w\in V_{div}$ we define the trilinear $V_{div}-$continuous form
$$b(u,v,w)=\int_{\Omega}(u\cdot\nabla)v\cdot w,$$
and the bilinear operator $\mathcal{B}$ from $V_{div}\times V_{div}$ into $V_{div}'$ defined by
$$\langle\mathcal{B}(u,v),w\rangle=b(u,v,w),\qquad\forall u,v,w\in V_{div}.$$
We recall that we have
\begin{equation}
b(u,w,v)=-b(u,v,w),\qquad\forall u,v,w\in V_{div},\label{ftril1}
\end{equation}
and that, for every $u$, $v$ and $w\in V_{div}$, the following
estimates hold
\begin{eqnarray}
& &|b(u,v,w)|\leq c\|u\|^{1/2}\|\nabla u\|^{1/2}\|\nabla v\|\|\nabla w\|,\quad\mbox{for }d=3,\label{ftril2}\\
& &|b(u,v,w)|\leq c\|u\|^{1/2}\|\nabla u\|^{1/2}\|\nabla v\|\|w\|^{1/2}\|\nabla w\|^{1/2},\quad\mbox{for }d=2.
\label{ftril3}
\end{eqnarray}

In this paper $c$ will stand for a nonnegative constant depending possibly only on $J$, $f$, $\Omega$, $\nu$ and $T$. The value of $c$ may vary even within the same line.
We shall denote by $N$, $M$ or $L$ generic nonnegative constants that depend on the initial
data $u_0$, $\varphi_0$ and on $h$ and whose values will be
explicitly pointed out if needed.


\section{Main result}\setcounter{equation}{0}
In this section we first define the notion of weak solution to problem
\eqref{sy1}-\eqref{sy6} which will be called Problem \textbf{P}.
Then we state the main result of this paper and a related corollary.

Our assumptions on the kernel $J$, the potential $F$ and the forcing
term $h$ are the following (cf. also \eqref{defp})
\begin{description}
\item[(H1)] $J\in W^{1,1}(\mathbb{R}^d),\quad
    J(x)=J(-x),\quad a(x) := \displaystyle
\int_{\Omega}J(x-y)dy \geq 0,\quad\mbox{a.e. } x\in\Omega$.
\item[(H2)] $F\in C^2(\mathbb{R})$ and there exists $c_0>0$
    such that
            $$F^{\prime\prime}(s)+a(x)\geq c_0,\qquad\forall s\in\mathbb{R},\quad\mbox{a.e. }x\in\Omega.$$
\item[(H3)] There exist
    $c_1>\frac{1}{2}\|J\|_{L^1(\mathbb{R}^d)}$ and
    $c_2\in\mathbb{R}$ such that
            $$F(s)\geq c_1s^2-c_2,\qquad\forall s\in\mathbb{R}.$$
\item[(H4)] There exist $c_3>0$, $c_4\geq0$ and $p\in(1,2]$
    such that
            $$|F^\prime(s)|^p\leq c_3|F(s)|+c_4,\qquad
            \forall s\in\mathbb{R}.$$
\item[(H5)] $h\in L_{loc}^2(0,T;V_{div}')$ for all $T>0$.
\end{description}

\begin{oss}
{\upshape The requirements of assumption (H1) are standard for the
nonlocal Cahn-Hilliard equation (see, e.g., \cite{BH1} for slightly
stronger hypotheses).}
\end{oss}

\begin{oss}
\label{g} {\upshape Assumption (H2) implies that the potential $F$
is a quadratic perturbation of a (strictly) convex function. Indeed, if
we set $a^{\ast}:=\|a\|_{\infty}$, then $F$ can be represented as
\begin{equation}
F(s)=G(s)-\frac{a^{\ast}}{2}s^2,\label{convex}
\end{equation}
with $G\in C^2(\mathbb{R})$ strictly convex, since $G''\geq c_0$ in
$\Omega$. }
\end{oss}

\begin{oss}
\label{dw}{\upshape Assumption (H4) is fulfilled by a
potential of arbitrary polynomial growth. In particular, (H2)-(H4) are
satisfied for the case of the physically relevant double-well potential,
i.e.
$$F(s)=(1-s^2)^2.$$
In this case we take $p=4/3$ in (H4), while assumption (H2) is
satisfied if and only if we have $p\geq c_0+m_0$, where
$m_0=-\min_{s\in\mathbb{R}}F^{\prime\prime}(s)$.}
\end{oss}

By weak solution we mean

\begin{defn}
\label{wfdfn} Let $u_0\in G_{div}$, $\varphi_0\in H$ with
$F(\varphi_0)\in L^1(\Omega)$ and $0<T<+\infty$ be given. Then
$[u,\varphi]$ is a weak solution to Problem \textbf{P} on $[0,T]$
corresponding to $u_0$ and $\varphi_0$ if
\begin{itemize}
\item $u$, $\varphi$ and $\mu$ satisfy
\begin{eqnarray}
& &u\in L^{\infty}(0,T;G_{div})\cap L^2(0,T;V_{div}),\label{df1}\\
& &u_t\in L^{4/3}(0,T;V_{div}'),\qquad\mbox{if}\quad d=3,\\
& &u_t\in L^{2-\gamma}(0,T;V_{div}'),\qquad\forall\gamma\in(0,1),\quad\mbox{if}\quad d=2,\\
& &\varphi\in L^{\infty}(0,T;H)\cap L^2(0,T;V),\\
& &\varphi_t\in L^{4/3}(0,T;V'),\quad\mbox{if}\quad d=3,\label{df5}\\
& &\varphi_t\in L^{2-\delta}(0,T;V'),
\quad\forall\delta\in(0,1),\quad\mbox{if}\quad d=2,\label{s}\label{df6}\\
& & \mu\in L^2(0,T;V).
\end{eqnarray}
\item setting
\begin{equation}
\rho(x,\varphi):=a(x)\varphi+F^\prime(\varphi),\label{h}
\end{equation}
then, for every $\psi\in V$, every $v\in V_{div}$ and for almost
any $t\in(0,T)$ we have
\begin{eqnarray}
& &\langle\varphi_t,\psi\rangle+(\nabla\rho,\nabla\psi)=\int_{\Omega}(u\cdot\nabla\psi)\varphi
+\int_{\Omega}(\nabla J\ast\varphi)\cdot\nabla\psi,\label{wf1}\\
& &\langle u_t,v\rangle+\nu(\nabla u,\nabla v)+b(u,u,v)=-\int_{\Omega}(v\cdot\nabla\mu)\varphi
+\langle h,v\rangle.\label{wf2}
\end{eqnarray}
\item the following initial conditions hold
\begin{equation}
u(0)=u_0,\qquad\varphi(0)=\varphi_0.\label{ic}
\end{equation}
\end{itemize}
\end{defn}

\begin{oss}
{\upshape
Since $\rho=\mu+J\ast\varphi$, from Definition \ref{wfdfn} we have that $\rho\in L^2(0,T;V)$.}
\end{oss}

\begin{oss}
\label{mass} {\upshape It is immediate to see that the total mass is
conserved. Indeed, choosing $\psi=1$ in \eqref{wf1}, we have
$\langle\varphi_t,1\rangle=0$ whence $(\varphi(t),1)=
(\varphi_0,1)$ for all $t\geq 0$.}
\end{oss}

\begin{oss}
{\upshape
The initial conditions \eqref{ic} are meant in the weak sense, i.e.,
for every $v\in V_{div}$ we have $(u(t),v)\to(u_0,v)$ as $t\to 0$,
and for every $\chi\in V$ we have $(\varphi(t),\chi)\to(\varphi_0,\chi)$ as $t\to 0$.
It can be proved that $u\in C_w([0,T];G_{div})$ and $\varphi\in C_w([0,T];H)$.
}
\end{oss}

\begin{thm}
\label{thm} Let $u_0\in G_{div}$, $\varphi_0\in H$ such that
$F(\varphi_0)\in L^1(\Omega)$ and suppose that (H1)-(H5) are
satisfied. Then, for every $T>0$ there exists a weak solution
$[u,\varphi]$ to Problem \textbf{P} on $[0,T]$ corresponding to
$u_0$, $\varphi_0$ with $\varphi_t$ satisfying
\begin{eqnarray}
& &\varphi_t\in L^{\infty}(0,T;V_s'),
\quad\mbox{if}\quad 1<p<\frac{d}{d-1},\:
s=\frac{(4-d)p+2d}{2p},\nonumber\\
& &\varphi_t\in L^{\infty}(0,T;V_s')\cap L^r(0,T;V_{\frac{d+2}{2}}'),\quad\mbox{if}\quad p=\frac{d}{d-1},\:\:
s>\frac{d+2}{2},\:\: r\geq 2,\nonumber\\
& &\varphi_t\in L^{2p/(2p-3)}(0,T;V_s'),\quad\mbox{if}\quad d=3,\:\: 3/2<p\leq 2,\:\: s=\frac{p+6}{2p}.
\nonumber
\end{eqnarray}
Furthermore, setting
$$\mathcal{E}(u(t),\varphi(t))=\frac{1}{2}\|u(t)\|^2+\frac{1}{4}
\int_{\Omega}\int_{\Omega}J(x-y)(\varphi(x,t)-\varphi(y,t))^2 dxdy+\int_{\Omega}F(\varphi(t))$$
the following energy
inequality holds for almost any $t>0$
\begin{equation}
\mathcal{E}(u(t),\varphi(t)) +\int_0^t(\nu\|\nabla u(\tau)\|^2+\|\nabla\mu(\tau)\|^2)d\tau
\leq\mathcal{E}(u_0,\varphi_0)  +\int_0^t\langle h(\tau),u(\tau)\rangle d\tau.\label{ei}
\end{equation}
\end{thm}

On account of the typical examples of double-well smooth potentials (cf. Remark \ref{dw}), the following additional assumption sounds reasonable (see, e.g., \cite[(A2)]{BH1})

\medskip
\begin{description}
\item[(H6)] $F\in C^2(\mathbb{R})$ and there exist $c_5>0$, $c_6>0$ and $q>0$
    such that
            $$F^{\prime\prime}(s)+a(x)\geq c_5\vert s\vert^{2q} - c_6,
            \qquad\forall s\in\mathbb{R},\quad\mbox{a.e. }x\in\Omega.$$
\end{description}
This requirement can replace (H3) in the proof of Theorem \ref{thm} (see
\eqref{coerc} below). Indeed, (H6) implies the existence of $c_7>0$
and $c_8>0$ such that
\begin{equation}
\label{coerc}
F(s)\geq c_7|s|^{2+2q}-c_8,\quad\forall s\in\mathbb{R}.
\end{equation}
Moreover, (H6) leads to establish further regularity properties for $\varphi$, $\varphi_t$, $u_t$.
This is stated in the following

\begin{cor}
\label{cor}
Suppose that the assumptions of Theorem \ref{thm} with (H3) replaced by (H6).
Then, for every $T>0$ there exists a weak solution
$[u,\varphi]$ to Problem \textbf{P} on $[0,T]$ corresponding to
$[u_0,\varphi_0]$ such that
\begin{eqnarray}
& &\varphi \in L^\infty(0,T;L^{2+2q}(\Omega)), \label{impr0}\\
& &\varphi_t\in L^2(0,T;V'),\quad\mbox{if}\quad d=2\quad\mbox{ or }\quad d=3 \mbox{ and } q\geq 1/2,\label{impr2}\\
& & u_t\in L^2(0,T;V_{div}'),\quad\mbox{if}\quad d=2,\label{u_tnew}
\end{eqnarray}
and
\begin{eqnarray}
& &\varphi_t\in L^{\infty}(0,T;V_s'),\quad\mbox{if }
\left\{\begin{array}{ll}
d=2,3,\quad 1<p\leq \frac{d}{d-1},\\
d=3,\quad 3/2<p\leq 2,\quad q\geq\frac{2(2p-3)}{6-p},
\end{array} \right.\label{impr1}\\
& &\varphi_t\in L^{\sigma}(0,T;V_s'),\quad\mbox{if}\quad d=3,\quad 3/2<p\leq 2,\quad 0<q<\frac{2(2p-3)}{6-p},
\label{beta}
\end{eqnarray}
where $s=((4-d)p+2d)/2p$ and in \eqref{beta} the exponent $\sigma$
is given by
$$\sigma=\frac{2p(1-\frac{q}{2})}{(2p-3)-q(3-\frac{p}{2})}.$$
\end{cor}

In two dimensions, as further consequences of (H6), we can prove
the energy identity and a dissipative estimate,
provided that $h\in L^2(0,\infty; V_{div}')$. Indeed, we have
\begin{cor}
\label{cor2}
Let $d=2$ and suppose that the assumptions of Theorem \ref{thm} with (H3) replaced by (H6) hold.
Then the weak solution $[u,\varphi]$ to Problem \textbf{P} corresponding to
$[u_0,\varphi_0]$ satisfies
\begin{equation}
\frac{d}{dt}\mathcal{E}(u,\varphi)+\nu\|\nabla u\|^2+\|\nabla\mu\|^2=\langle h,u\rangle.
\label{idendiffcor}
\end{equation}
Therefore, \eqref{ei} with the equal sign
holds for every $t\geq 0$.
Furthermore, if in addition $h\in L^2(0,\infty; V_{div}')$,
then the following dissipative estimate is satisfied
\begin{equation}
\mathcal{E}(u(t),\varphi(t))\leq \mathcal{E}(u_0,\varphi_0)e^{-kt}+ F(m)|\Omega| + K,
\qquad\forall t\geq 0,\label{dissest}
\end{equation}
where $m=(\varphi_0,1)$ and $k$, $K$ are two positive constants
which are independent of the initial data, with $K$ depending on
$\Omega$, $\nu$, $J$, $F$, $\|h\|_{L^2(0,\infty;V_{div}')}$.
\end{cor}

\begin{oss}
{\upshape
It follows from Corollary \ref{cor2} that, in two dimensions, $u\in C([0,T];G_{div})$ and $\varphi\in C([0,T];H)$.
This fact along with the validity of an energy identity suggests that the generalized
semiflow approach devised in \cite{Ba} (see also \cite{MRW}) might be applied to our system. If so, one should be able to establish the existence of a global attractor. This is one of the issues which will be investigated in a forthcoming paper.
}
\end{oss}

\section{Proof of Theorem \ref{thm}}\setcounter{equation}{0}
The proof will be carried out by means of a Faedo-Galerkin
approximation scheme. We will assume first that $\varphi_0\in D(B)$.
The existence under the stated assumption on $\varphi_0$ will be
recovered by a density argument by exploiting the form of the
potential $F$ as a quadratic perturbation of a convex function (see
Remark \ref{g}).

We introduce the family $\{w_j\}_{j\geq 1}$ of the eigenfunctions of the Stokes operator $A$ as a Galerkin
base in $V_{div}$ and the family $\{\psi_j\}_{j\geq 1}$ of the eigenfunctions of the Neumann operator
$$B=-\Delta+I$$
as a Galerkin base in $V$. We define the $n-$dimensional subspaces
$\mathcal{W}_n:=\langle w_1,\cdots,w_n\rangle$ and
$\Psi_n:=\langle\psi_1,\cdots,\psi_n\rangle$ and consider the
orthogonal projectors on these subspaces in $G_{div}$ and $H$,
respectively, i.e., $\widetilde{P}_n:=P_{\mathcal{W}_n}$ and
$P_n:=P_{\Psi_n}$. We then look for three functions of the form
$$u_n(t)=\sum_{k=1}^n a^{(n)}_k(t)w_k,\quad \varphi_n(t)=\sum_{k=1}^n b^{(n)}_k(t)\psi_k,
\quad \mu_n(t)=\sum_{k=1}^n c^{(n)}_k(t)\psi_k$$ which solve
to the following approximating problem

\begin{eqnarray}
& &(\varphi_n',\psi)+(\nabla \rho(\cdot,\varphi_n),\nabla\psi)
=\int_{\Omega}(u_n\cdot\nabla\psi)\varphi_n+\int_{\Omega}(\nabla J\ast\varphi_n)\cdot\nabla\psi\label{ap1}\\
& &(u_n',w)+\nu(\nabla u_n,\nabla w)+b(u_n,u_n,w)=-\int_{\Omega}(w\cdot\nabla\mu_n)\varphi_n
+(h_n,w)\label{ap2}\\
& &\rho(\cdot,\varphi_n):=a(\cdot)\varphi_n+F^\prime(\varphi_n)\label{ap3}\\
& &\mu_n=P_n(p\varphi_n+f(\varphi_n)-J\ast\varphi_n)\label{ap4}\\
& &\varphi_n(0)=\varphi_{0n},\quad u_n(0)=u_{0n},\label{ap5}
\end{eqnarray}
for every $\psi\in\Psi_n$ and every $w\in\mathcal{W}_n$, where
$\varphi_{0n}=P_n\varphi_0$ and $u_{0n}=\widetilde{P}_nu_0$
(primes denote derivatives with respect to time). In \eqref{ap2}
$h_n\in C^0([0,T];G_{div})$ and, on account of (H5), we choose the
sequence of $h_n$ in such a way that $h_n\to h$ in
$L^2(0,T;V_{div}')$. It is easy to see that this approximating problem
is equivalent to solving a Cauchy problem for a system of ordinary
differential equations in the $2n$ unknowns $a^{(n)}_i$,
$b^{(n)}_i$. Since $F^\prime\in C^1(\mathbb{R})$, the
Cauchy-Lipschitz theorem ensures that there exists
$T^{\ast}_n\in(0,+\infty]$ such that this system has a unique
maximal solution
$\textbf{a}^{(n)}:=(a^{(n)}_1,\cdots,a^{(n)}_n)$,
$\textbf{b}^{(n)}:=(b^{(n)}_1,\cdots,b^{(n)}_n)$ on
$[0,T^{\ast}_n)$ and $\textbf{a}^{(n)}$, $\textbf{b}^{(n)}\in
C^1([0,T^{\ast}_n);\mathbb{R}^n)$.


We now
    derive some a priori estimates in order to show that
    $T^{\ast}_n=+\infty$ for every $n\geq 1$ and that the
    sequences of $\varphi_n$, $u_n$ and $\mu_n$ are bounded in
    suitable functional spaces. By using $\mu_n$ as a test function in
    \eqref{ap1}, $u_n$ as  a test function in \eqref{ap2} and
    recalling that $b(u_n,u_n,u_n)=0$ (see \eqref{ftril1}), we obtain
\begin{eqnarray}
& &(\varphi_n',\mu_n)+(\nabla \rho(\cdot,\varphi_n),\nabla\mu_n)=\int_{\Omega}(u_n\cdot\nabla\mu_n)\varphi_n+\int_{\Omega}(\nabla J\ast\varphi_n)\cdot\nabla\mu_n\nonumber\\
& &\frac{1}{2}\frac{d}{dt}\|u_n\|^2+\nu\|\nabla u_n\|^2=-\int_{\Omega}(u_n\cdot\nabla\mu_n)\varphi_n
+(h_n,u_n).\nonumber
\end{eqnarray}
We now have
\begin{eqnarray}
& &(\varphi_n',\mu_n)=(\varphi_n',a\varphi_n+F^\prime(\varphi_n)-J\ast\varphi_n)\nonumber\\
& &=\frac{d}{dt}\Big(\frac{1}{2}\|\sqrt{a}\varphi_n\|^2+\int_{\Omega}F(\varphi_n)-\frac{1}{2}(\varphi_n,J\ast\varphi_n)\Big)\nonumber\\
& &=\frac{d}{dt}\Big(\frac{1}{4}\int_{\Omega}\int_{\Omega}J(x-y)(\varphi_n(x)-\varphi_n(y))^2 dxdy
+\int_{\Omega}F(\varphi_n)\Big).
\end{eqnarray}
Furthermore, observe that
\begin{eqnarray}
& &(\nabla \rho(\cdot,\varphi_n),\nabla\mu_n)=(-\rho(\cdot,\varphi_n),\Delta\mu_n)=(-\rho_n,\Delta\mu_n)
=(\nabla \rho_n,\nabla\mu_n),\nonumber
\end{eqnarray}
where $\rho_n:=P_n
\rho(\cdot,\varphi_n)=\mu_n+P_n(J\ast\varphi_n)$. Summing the
first two identities and taking the previous relations into account we
get
\begin{eqnarray}
& &\frac{1}{2}\frac{d}{dt}\Big(\|u_n\|^2+\frac{1}{2}\int_{\Omega}\int_{\Omega}J(x-y)(\varphi_n(x)-\varphi_n(y))^2 dxdy
+2\int_{\Omega}F(\varphi_n)\Big)\nonumber\\
& &+\nu\|\nabla u_n\|^2+\|\nabla\mu_n\|^2+(\nabla(P_n(J\ast\varphi_n)),\nabla\mu_n)\nonumber\\
& &=\int_{\Omega}(\nabla J\ast\varphi_n)\cdot\nabla\mu_n+(h_n,u_n).
\label{est1}
\end{eqnarray}
Now, it is easy to see that
\begin{eqnarray}
& &\|\nabla(P_n(J\ast\varphi_n))\|\leq\|B^{1/2}P_n(J\ast\varphi_n)\|\nonumber\\
& &\leq\|\nabla J\ast\varphi_n\|+\|J\ast\varphi_n\|\leq\|J\|_{W^{1,1}}\|\varphi_n\|,\label{w1}
\end{eqnarray}
and that, by means of (H3), we have
\begin{eqnarray}
& &\frac{1}{2}\int_{\Omega}\int_{\Omega}J(x-y)(\varphi_n(x)-\varphi_n(y))^2 dxdy
+2\int_{\Omega}F(\varphi_n)\nonumber\\
& &=\|\sqrt{a}\varphi_n\|^2+2\int_{\Omega}F(\varphi_n)-(\varphi_n,J\ast\varphi_n)\nonumber\\
& &\geq\int_{\Omega}(a+2c_1-\|J\|_{L^1})\varphi_n^2-2c_2|\Omega|\geq\alpha\|\varphi_n\|^2-c,
\label{w2}
\end{eqnarray}
where $\alpha=2c_1-\|J\|_{L^1}>0$. Hence, integrating
\eqref{est1} with respect to time between $0$ and
$t\in(0,T^{\ast}_n)$ and using \eqref{w1}, \eqref{w2}, we are led
to the following differential inequalities
\begin{eqnarray}
& &\|u_n\|^2+\alpha\|\varphi_n\|^2+\int_0^t(\nu\|\nabla u_n\|^2+\|\nabla\mu_n\|^2)d\tau
\leq c\|J\|_{W^{1,1}}^2\int_0^t\|\varphi_n\|^2 d\tau\nonumber\\
& &+\|u_{0n}\|^2+\frac{1}{2}\int_{\Omega}\int_{\Omega}J(x-y)(\varphi_{0n}(x)-\varphi_{0n}(y))^2 dxdy
\nonumber\\
& &+2\int_{\Omega}F(\varphi_{0n})+\frac{1}{2\nu}\int_0^t \|h_n\|_{V_{div}'}^2d\tau+c\nonumber\\
& &\leq M+\frac{1}{2\nu}\int_0^t\|h\|_{V_{div}'}^2d\tau+c\int_0^t\|\varphi_n\|^2 d\tau,\qquad\forall t\in[0,T^{\ast}_n),\label{stb}
\end{eqnarray}
where $c$ only depends on $\|J\|_{W^{1,1}}$ and on $|\Omega|$, while $M$ is given by
\begin{equation}
M=c\Big(1+\|u_0\|^2+\|\varphi_0\|^2+\int_{\Omega}F(\varphi_0)\Big).
\end{equation}
Here we have used the fact that, since $\varphi_0$ is supposed to
belong to $D(B)$, then we have $\varphi_{0n}\to\varphi_0$ in
$H^2(\Omega)$ and hence also in $L^{\infty}(\Omega)$ (for
$d=2,3$). Since we have $\|u_n(t)\|=|\textbf{a}^{(n)}(t)|$ and
$\|\varphi_n(t)\|=|\textbf{b}^{(n)}(t)|$, by means of Gronwall
lemma we deduce that $T^{\ast}_n=+\infty$, for every $n\geq1$,
i.e., problem \eqref{ap1}-\eqref{ap5} has a unique global in time
solution, and that \eqref{stb} is satisfied for every $t\geq0$.
Furthermore, we obtain the following estimates holding for any given
$0<T<+\infty$
\begin{eqnarray}
& &\|u_n\|_{L^{\infty}(0,T;G_{div})\cap L^2(0,T;V_{div})}\leq N,\label{e1}\\
& &\|\varphi_n\|_{L^{\infty}(0,T;H)}\leq N,\label{e2}\\
& &\|\nabla\mu_n\|_{L^2(0,T;H)}\leq N,\label{e3}
\end{eqnarray}
where
$$N=cM^{1/2}+c\|h\|_{L^2(0,T;V_{div}')},$$
with $c$ now depending also on $T$
and on $\nu$.
From \eqref{ap4}, \eqref{e3} and recalling
\eqref{defp} we now deduce an estimate for $\varphi_n$ in $L^2(0,T;V)$.
We have
\begin{eqnarray}
& &(\mu_n,-\Delta\varphi_n)=(\nabla\mu_n,\nabla\varphi_n)=
(-\Delta\varphi_n,a\varphi_n+F^\prime(\varphi_n)-J\ast\varphi_n)\nonumber\\
& &=(\nabla\varphi_n,a\nabla\varphi_n+\varphi_n\nabla a+
F^{\prime\prime}(\varphi_n)\nabla\varphi_n-\nabla J\ast\varphi_n)\nonumber\\
& &\geq c_0\|\nabla\varphi_n\|^2-2\|\nabla J\|_{L^1}\|\nabla\varphi_n\|\|\varphi_n\|\nonumber\\
& &\geq\frac{c_0}{2}\|\nabla\varphi_n\|^2-k\|\varphi_n\|^2,
\end{eqnarray}
where $k=(2/c_0)\|\nabla J\|_{L^1}^2$ and where we have used (H2). Since
$$(\nabla\mu_n,\nabla\varphi_n)\leq\frac{c_0}{4}\|\nabla\varphi_n\|^2+\frac{1}{c_0}\|\nabla\mu_n\|^2,$$
we get
\begin{equation}
\|\nabla\mu_n\|^2\geq\frac{c_0^2}{4}\|\nabla\varphi_n\|^2-c\|\varphi_n\|^2,\label{y}
\end{equation}
and \eqref{e2}, \eqref{e3}, \eqref{y} yield
\begin{equation}
\|\varphi_n\|_{L^2(0,T;V)}\leq N.\label{e5}
\end{equation}
The next step is to deduce an estimate for the sequence of $\mu_n$
in $L^2(0,T;V)$. To this aim we first observe that (H4) implies that
$|F^\prime(s)|\leq c|F(s)|+c$ for every $s\in\mathbb{R}$ and
therefore we have
\begin{eqnarray}
& &\Big|\int_{\Omega}\mu_n\Big|=|(\mu_n,1)|=|(a\varphi_n+F^\prime(\varphi_n)-J\ast\varphi_n,1)|\nonumber\\
& &\leq\int_{\Omega}|F^\prime(\varphi_n)|+c\leq c\int_{\Omega}|F(\varphi_n)|+c\leq N,\label{e4}
\end{eqnarray}
since we have $\|F(\varphi_n)\|_{L^{\infty}(0,T;L^1(\Omega))}\leq N$ due to \eqref{est1}
(integrated in time between $0$ and $t\in[0,T]$) and \eqref{w2}. We have also used the estimates \eqref{e1}-\eqref{e3}. Hence, by means of the Poincar\'{e}-Wirtinger inequality,
from \eqref{e3} and \eqref{e4} we get
\begin{equation}
\|\mu_n\|_{L^2(0,T;V)}\leq N.\label{e6}
\end{equation}
We also need an estimate for the sequence
$\{\rho(\cdot,\varphi_n)\}$. From (H4) we immediately get
\begin{equation}
 \|\rho(\cdot,\varphi_n)\|_{L^p}\leq(c\|a\|_{L^{\infty}}\|\varphi_n\|+\|F^\prime(\varphi_n)\|_{L^p})\nonumber\\
 \leq c\Big(\Big(\int_{\Omega}|F(\varphi_n)|\Big)^{1/p}+1\Big)\leq N,\label{e9b}
\end{equation}
and hence we have
\begin{equation}
\|\rho(\cdot,\varphi_n)\|_{L^{\infty}(0,T;L^p(\Omega))}\leq N.\label{wg1}
\end{equation}



    The final estimates we need are for the
    sequences of time derivatives $u_n'$ and $\varphi_n'$. Let us
    start from the sequence of $u_n'$. Equation \eqref{ap2} can be
    written as
\begin{equation}
u_n'+\nu Au_n+\widetilde{P}_n\mathcal{B}(u_n,u_n)=-\widetilde{P}_n(\varphi_n\nabla\mu_n)+\widetilde{P}_n h_n\label{td1}
\end{equation}
We now have, for $d=3$, by using Sobolev embeddings, interpolation
between $L^p$ spaces and \eqref{e2}
\begin{equation}
 \|\widetilde{P}_n(\varphi_n\nabla\mu_n)\|_{V_{div}'}\leq c\|\varphi_n\|_{L^3}\|\nabla\mu_n\|
 \leq c\|\varphi_n\|^{1/2}\|\varphi_n\|_{L^6}^{1/2}\|\nabla\mu_n\|\leq N^{1/2}\|\varphi_n\|_V^{1/2}\|\nabla\mu_n\|.
 \label{td1bis}
\end{equation}
Therefore, thanks to \eqref{e3} and \eqref{e5}, we get
\begin{equation}
\|\widetilde{P}_n(\varphi_n\nabla\mu_n)\|_{L^{4/3}(0,T;V_{div}')}\leq N^2.\label{td2}
\end{equation}
For the case $d=2$, by means of Gagliardo-Nirenberg interpolation inequality in dimension 2 we have,
for every $0<\gamma<1$
\begin{eqnarray}
& &\|\widetilde{P}_n(\varphi_n\nabla\mu_n)\|_{V_{div}'}\leq c\|\varphi_n\|_{L^{2+\gamma/(1-\gamma)}}\|\nabla\mu_n\|
\nonumber\\
& &\leq c\|\varphi_n\|^{2(1-\gamma)/(2-\gamma)}\|\varphi_n\|_V^{\gamma/(2-\gamma)}\|\nabla\mu_n\|\nonumber\\
& &\leq N^{2(1-\gamma)/(2-\gamma)}\|\varphi_n\|_V^{\gamma/(2-\gamma)}\|\nabla\mu_n\|,\label{td33}
\end{eqnarray}
so that \eqref{e3} and \eqref{e5} yield
\begin{equation}
\|\widetilde{P}_n(\varphi_n\nabla\mu_n)\|_{L^{2-\gamma}(0,T;V_{div}')}\leq N^2.\label{td4}
\end{equation}
Moreover, we have $\|Au_n\|_{V_{div}'}=\|u_n\|_{V_{div}}$,
while the treatment of the term $\widetilde{P}_n
\mathcal{B}(u_n,u_n)$ is classical and, by means of \eqref{ftril2}
and \eqref{ftril3} we have
\begin{eqnarray}
& &\|\widetilde{P}_n \mathcal{B}(u_n,u_n)\|_{V_{div}'}\leq c\|u_n\|^{1/2}\|u_n\|_{V_{div}}^{3/2},
\qquad\mbox{for }d=3,\label{td3}\\
& &\|\widetilde{P}_n \mathcal{B}(u_n,u_n)\|_{V_{div}'}\leq
c\|u_n\|\|u_n\|_{V_{div}},\qquad\mbox{for }d=2.\label{td5}
\end{eqnarray}
Hence, by using \eqref{td2}, \eqref{td4} and \eqref{td3}, \eqref{td5},
and recalling that $\widetilde{P}_n\in\mathcal{L}(V_{div}',V_{div}')$, which implies that
$$\|\widetilde{P}_n h_n\|_{L^2(0,T;V_{div}')}\leq c(1+\|h\|_{L^2(0,T;V_{div}')}),$$
from \eqref{td1} we obtain
\begin{eqnarray}
& &\|u_n'\|_{L^{4/3}(0,T;V_{div}')}\leq L,\qquad\mbox{for }d=3,\label{e7}\\
& &\|u_n'\|_{L^{2-\gamma}(0,T;V_{div}')}\leq L,\qquad\forall\gamma\in(0,1),\quad\mbox{for }d=2,\label{wg2}
\end{eqnarray}
where $L=N^2+N$.

In order to derive an estimate for the sequence of $\varphi_n'$, we aim to take the test
function $\psi\in V_s$ in \eqref{ap1}, where $s\geq 2$ is such that $\Delta\psi\in H^{s-2}(\Omega)
\hookrightarrow L^{p'}(\Omega)$ ($p'$ is the conjugate index to $p$). Since $H^{s-2}\hookrightarrow L^{p^{\ast}}$,
where $p^{\ast}=2d/(d+4-2s)$, we see that it is enough to take
\begin{equation}
s\geq\frac{(4-d)p+2d}{2p}.\label{Vs}
\end{equation}
Let us now decompose $\psi$ as
$$\psi=\psi_{I}+\psi_{II},$$
where $\psi_{I}=P_n\psi=\sum_{k=1}^n(\psi,\psi_k)\psi_k\in\Psi_n$ and $\psi_{II}=(I-P_n)\psi=\sum_{k=n+1}^{\infty}(\psi,\psi_k)\psi_k\in\Psi_n^{\perp}$ (recall that $\psi_{I}$ and
$\psi_{II}$ are orthogonal in all the Hilbert spaces $V_r$, for every $0\leq r\leq s$), and notice that we have,
due to \eqref{e9b}
\begin{eqnarray}
& &|(\nabla \rho(\cdot,\varphi_n),\nabla\psi_{I})|=|(\rho(\cdot,\varphi_n),\Delta\psi_{I})|\nonumber\\
& &\leq N\|\Delta\psi_{I}\|_{L^{p'}}\leq N\|\psi_{I}\|_{V_s}\nonumber\\
& &\leq N\|\psi\|_{V_s}.\label{derphi1}
\end{eqnarray}
Furthermore, it is easy to see that
\begin{eqnarray}
& &\Big|\int_{\Omega}(\nabla J\ast\varphi_n)\cdot\nabla\psi_{I}\Big|\leq c\|\nabla J\|_{L^1}\|\varphi_n\|\|\psi\|_{V_s}\leq N\|\psi\|_{V_s}.
\end{eqnarray}
As far as the first term in the right hand side of \eqref{ap1} (written with $\psi=\psi_{I}$) is concerned
we notice that $\nabla\psi_{I}\in H^{s-1}(\Omega)$.
Therefore, when $1<p<d/(d-1)$ and $s=((4-d)p+2d)/2p$ or $p=d/(d-1)$ and $s>((4-d)p+2d)/2p=(d+2)/2$,
due to the embedding $H^{s-1}\hookrightarrow L^{\infty}$, we have
\begin{equation}
\Big|\int_{\Omega}(u_n\cdot\nabla\psi_{I})\varphi_n\Big|\leq c\|u_n\|\|\varphi_n\|\|\psi\|_{V_s}\leq N^2\|\psi\|_{V_s}.\label{derphi0}
\end{equation}
When $p=d/(d-1)$ and $s=((4-d)p+2d)/2p=(d+2)/2$, due to the
embedding $H^{s-1}\hookrightarrow L^q$ for every $1\leq
q<+\infty$ and interpolation in $L^p$ spaces, we have, for every
$r\geq 2$, that
\begin{eqnarray}
& &
\Big|\int_{\Omega}(u_n\cdot\nabla\psi_{I})\varphi_n\Big|\leq c\|u_n\|\|\psi\|_{V_s}\|\varphi_n\|_{L^{2r/(r-1)}}
\nonumber\\
& &\leq c\|u_n\|\|\psi\|_{V_s}\|\varphi_n\|^{(r-2)/r}\|\varphi_n\|_{L^4}^{2/r}
\leq N^{2/r'}\|\psi\|_{V_s}\|\varphi_n\|_V^{2/r}.
\end{eqnarray}
Finally, in the case $d=3$, when $3/2<p\leq 2$ and
$s=((4-d)p+2d)/2p=(p+6)/2p$, due to the embedding
$H^{s-1}\hookrightarrow L^{3p/(2p-3)}$, we obtain
\begin{eqnarray}
& &\Big|\int_{\Omega}(u_n\cdot\nabla\psi_{I})\varphi_n\Big|\leq c\|u_n\|\|\psi\|_{V_s}\|\varphi_n\|_{L^{6p/(6-p)}}
\nonumber\\
& &\leq c\|u_n\|\|\psi\|_{V_s}\|\varphi_n\|^{(3-p)/p}\|\varphi_n\|_{L^6}^{(2p-3)/p}\nonumber\\
& &\leq N^{3/p}\|\psi\|_{V_s}\|\varphi_n\|_V^{(2p-3)/p}.\label{derphi2}
\end{eqnarray}
Collecting \eqref{derphi1}-\eqref{derphi2}, from \eqref{ap1} (written with $\psi=\psi_{I}$)
we then get
\begin{eqnarray}
& &\|\varphi_n'\|_{L^{\infty}(0,T;V_s')}\leq L\quad\mbox{if }1<p<d',\:
s=\frac{(4-d)p+2d}{2p},\label{dertphi1}\\
& &\|\varphi_n'\|_{L^{\infty}(0,T;V_s')\cap L^r(0,T;V_{\frac{d+2}{2}}')}\leq L\quad\mbox{if }p=d',\:
s>\frac{d+2}{2},\: r\geq 2,\label{dertphi2}
\end{eqnarray}
where $d'=d/(d-1)$, while in the case $d=3$, if $3/2<p\leq 2$ we find
\begin{equation}
\|\varphi_n'\|_{L^{2p/(2p-3)}(0,T;V_s')}\leq L,\quad  s=\frac{p+6}{2p},\label{dertphi3}
\end{equation}
where $L=N+N^2$ in all cases.

From the estimates
    \eqref{e1}-\eqref{e3}, \eqref{e5}, \eqref{e6}, \eqref{wg1},
    \eqref{e7}, \eqref{wg2}, \eqref{dertphi1}-\eqref{dertphi3}
    and on account of the compact embeddings
\begin{eqnarray}
& & L^2(0,T;V)\cap H^1(0,T;V_s')\hookrightarrow\hookrightarrow L^2(0,T;H),\nonumber\\
& & L^2(0,T;V_{div})\cap W^{1,q}(0,T;V_{div}')\hookrightarrow\hookrightarrow L^2(0,T;G_{div}),
\quad\forall q>1\nonumber
\end{eqnarray}
we deduce that there exist
\begin{eqnarray}
& & u\in L^{\infty}(0,T;G_{div})\cap L^2(0,T;V_{div}),\label{p1}\\
& & \varphi\in L^{\infty}(0,T;H)\cap L^2(0,T;V),\\
& & \mu\in L^2(0,T;V),\\
& & \rho\in L^{\infty}(0,T;L^p(\Omega))
\end{eqnarray}
with
\begin{eqnarray}
& &u_t\in L^{4/3}(0,T;V_{div}'),\qquad\mbox{if }d=3,\\
& &u_t\in L^{2-\gamma}(0,T;V_{div}'),\qquad\forall\gamma\in(0,1),\quad\mbox{if }d=2,\nonumber
\end{eqnarray}
and
\begin{eqnarray}
& &\varphi_t\in L^{\infty}(0,T;V_s'),
\quad\mbox{if }1<p<d',\:
s=\frac{(4-d)p+2d}{2p},\label{p2}\\
& &\varphi_t\in L^{\infty}(0,T;V_s')\cap L^r(0,T;V_{\frac{d+2}{2}}'),\quad\mbox{if }p=d',\:\:
s>\frac{d+2}{2},\:\: r\geq 2,\nonumber\\
& &\varphi_t\in L^{2p/(2p-3)}(0,T;V_s'),\quad\mbox{if }d=3,\:\: 3/2<p\leq 2,\:\: s=\frac{p+6}{2p},
\nonumber
\end{eqnarray}
such that, for a not relabeled subsequence, we deduce
\begin{eqnarray}
& & u_n\rightharpoonup u\quad\mbox{weakly}^{\ast}\mbox{ in } L^{\infty}(0,T;G_{div}),\label{c1}\\
& & u_n\rightharpoonup u\quad\mbox{weakly in }L^2(0,T;V_{div}),\label{c2}\\
& & u_n\to u\quad\mbox{strongly in }L^2(0,T;G_{div}),\quad\mbox{a.e. in }\Omega\times(0,T),\label{c3}\\
& & u_n'\rightharpoonup u_t\quad\mbox{weakly in }L^{4/3}(0,T;V_{div}'),\qquad d=3,\label{c4}\\
& & u_n'\rightharpoonup u_t
\quad\mbox{weakly in }L^{2-\gamma}(0,T;V_{div}'),\quad\forall\gamma\in(0,1),\:\: d=2,\label{yy}\\
& & \varphi_n\rightharpoonup\varphi\quad\mbox{weakly}^{\ast}\mbox{ in }L^{\infty}(0,T;H),\label{c5}\\
& & \varphi_n\rightharpoonup\varphi\quad\mbox{weakly in }L^2(0,T;V),\label{c6}\\
& & \varphi_n\to\varphi\quad\mbox{strongly in }L^2(0,T;H),\quad\mbox{a.e. in }\Omega\times(0,T),\label{c7}\\
& & \mu_n\rightharpoonup\mu\quad\mbox{weakly in }L^2(0,T;V),\label{c9}\\
& & \rho(\cdot,\varphi_n)\rightharpoonup \rho\quad\mbox{weakly}^{\ast}
\mbox{ in }L^{\infty}(0,T;L^p(\Omega)),\label{c10}
\end{eqnarray}
and
\begin{eqnarray}
& &\varphi_n'\rightharpoonup\varphi_t\quad\mbox{weakly}^{\ast}\mbox{ in }L^{\infty}(0,T;V_s'),
\end{eqnarray}
if $1<p<d'$, with $s=((4-d)p+2d)/2p$,
\begin{eqnarray}
& &\varphi_n'\rightharpoonup\varphi_t\:\mbox{weakly}^{\ast}\mbox{ in }L^{\infty}(0,T;V_s'),
\:\mbox{weakly in }L^r(0,T;V_{\frac{d+2}{2}}'),
\end{eqnarray}
if $p=d'$, with $s>(d+2)/2$ and $r\geq 2$,
\begin{eqnarray}
& &\varphi_n'\rightharpoonup\varphi_t\quad\mbox{weakly in }L^{2p/(2p-3)}(0,T;V_s'),
\end{eqnarray}
if $d=3$ and $3/2<p\leq 2$, with $s=(p+6)/2p$.

We can now pass to the limit in  \eqref{ap1}-\eqref{ap5} in order to
prove that the functions $u$ and $\varphi$ yield a weak solution to
Problem \textbf{P} in the sense of Definition 1, i.e., $u$, $\varphi$,
$\mu$ and $\rho$ satisfy \eqref{mu}, \eqref{h} and \eqref{wf1},
\eqref{wf2}, \eqref{ic}. First of all, from the pointwise convergence
\eqref{c7}  we have $\rho(\cdot,\varphi_n)\to
a\varphi+F^\prime(\varphi)$ almost everywhere in
$\Omega\times(0,T)$ and therefore from \eqref{c10} we have
$\rho=a\varphi+F^\prime(\varphi)$, i.e. \eqref{h}. Moreover, since
$\mu_k=P_k(\rho(\cdot,\varphi_k)-J\ast\varphi_k)$, we have, for
every $v\in\Psi_n$ and every $k\geq n$ ($n$ is fixed)
$$\int_0^T(\mu_k(t),v)\chi(t)dt=\int_0^T(\rho(\cdot,\varphi_k)-J\ast\varphi_k,v)\chi(t)dt,\quad\forall\chi\in C_0^{\infty}(0,T).$$
By passing to the limit as $k\to\infty$ in this identity and using the
convergences \eqref{c9}, \eqref{c7} (which implies
$J\ast\varphi_k\to J\ast\varphi$ strongly in $L^2(0,T;V)$) and
\eqref{c10}, on account of the density of $\{\Psi_n\}_{n\geq 1}$
in $H$ we get
$\mu=\rho-J\ast\varphi=a\varphi+F^\prime(\varphi)-J\ast\varphi$,
i.e. \eqref{mu}. In particular we obtain $\rho\in L^2(0,T;V)$.

The argument used to recover \eqref{wf1} and \eqref{wf2} by
passing to the limit in \eqref{ap1} and \eqref{ap2} of the
approximate problem and by exploiting the above convergences is
standard and we only limit ourselves to give a sketch of it. We multiply
\eqref{ap1} by $\chi$ and \eqref{ap2} by $\omega$, where $\chi$,
$\omega\in C_0^{\infty}(0,T)$ and integrate in time between 0 and
$T$. Due to the above convergences  we can pass to the limit in these
equations. In particular the term $(\nabla
\rho(\cdot,\varphi_n),\nabla\psi)$ can be rewritten as
$(\rho(\cdot,\varphi_n),-\Delta\psi)$ and \eqref{c10} is used. We
also recall that in the nonlinear term $b(u_n,u_n,w)\omega$ we
exploit the strong convergence \eqref{c3} to pass to the limit. The
limit equations thus obtained hold for every $\psi\in\Psi_n$, every
$w\in\mathcal{W}_n$ (where $n$ is fixed) and every $\chi$,
$\omega\in C_0^{\infty}(0,T)$. The density of $\{\Psi_n\}_{n\geq
1}$ and $\{\mathcal{W}_n\}_{n\geq 1}$ in $V_s$ and $V_{div}$,
respectively, allows us to conclude that $u$, $\varphi$, $\mu$ and
$\rho$ satisfy \eqref{wf1} for every $\psi\in V_s$ and \eqref{wf2}
for every $v\in V_{div}$. Furthermore,
observe that \eqref{wf1} can be written in the form
\begin{equation}
\langle\varphi_t,\psi\rangle=-(\nabla\mu,\nabla\psi)+(u,\varphi\nabla\psi),\label{phivar}
\end{equation}
and consider the contribution of the transport term in \eqref{phivar}. In the case $d=3$,
by arguing as in \eqref{td1bis} we have
\begin{eqnarray}
& &|(u,\varphi\nabla\psi)|\leq N^{1/2}\|\nabla u\|\|\varphi\|_V^{1/2}\|\nabla\psi\|,\label{d=3}
\end{eqnarray}
while, in the case $d=2$, by arguing as in \eqref{td33} we have
\begin{eqnarray}
& &|(u,\varphi\nabla\psi)|\leq
N^{2(1-\delta)/(2-\delta)}\|\nabla u\|\|\varphi\|_V^{\delta/(2-\delta)}\|\nabla\psi\|,\label{d=2}
\end{eqnarray}
for every $\delta\in(0,1)$.
From \eqref{d=3} and \eqref{d=2} we deduce that $\varphi_t(t)$ can be continuously extended to $V$ for
almost any $t>0$ and from these equations and \eqref{phivar} we also infer that
\begin{eqnarray}
& &\varphi_t\in L^{4/3}(0,T;V'),\quad\mbox{if}\quad d=3;\qquad\varphi_t\in L^{2-\delta}(0,T;V'),
\quad\forall\delta\in(0,1),\quad\mbox{if}\quad d=2.\nonumber
\end{eqnarray}
We hence get \eqref{df5}, \eqref{df6} and furthermore, \eqref{wf1} and \eqref{phivar} hold also for every
$\psi\in V$.


Finally, in order to get \eqref{ic}, it is enough to integrate \eqref{ap1}, \eqref{ap2}
between $0$ and $t$ and pass to the limit for $n\to\infty$ by using the weak
convergences above. By integrating between $t_0$ and $t$ we prove the weak
continuity of $u$ and $\varphi$ in $G_{div}$ and $H$, respectively.
$\\$

We now prove that the
    energy inequality \eqref{ei} holds for the weak solution
    $[u,\varphi]$ corresponding to the initial data $u_0\in G_{div}$
    and $\varphi_0\in D(B)$. To this aim let us first observe that, for
    almost any $t\in(0,T)$ and for a not relabeled subsequence we
    have
\begin{eqnarray}
& &u_n(t)\to u(t)\qquad\mbox{strongly in  }G_{div},\label{wq1}\\
& &\varphi_n(t)\to\varphi(t)\qquad\mbox{strongly in  }H\mbox{ and a.e. in  }\Omega.
\end{eqnarray}
and that, by means of (H3) and of Fatou's lemma we have
\begin{equation}
\int_{\Omega}F(\varphi(t))\leq\liminf_{n\to\infty}\int_{\Omega}F(\varphi_n(t)).\label{wq11}
\end{equation}
In addition, it is easy to see that
\begin{equation}
P_n(J\ast\varphi_n)\to J\ast\varphi\qquad\mbox{in }L^2(0,T;V),\label{wq2}
\end{equation}
as a consequence of the convergence $J\ast\varphi_n\to J\ast\varphi$
strongly in $L^2(V)$ and
of the fact that $P_n\in\mathcal{L}(V,V)$.
Hence, by integrating \eqref{est1} between 0 and $t$, and by passing to the limit
using \eqref{wq1}-\eqref{wq2},
the weak convergences \eqref{c2}, \eqref{c9} and the weak lower semicontinuity of the norm,
we immediately get \eqref{ei}.


In order to complete the proof of the theorem we now assume that $u_0\in G_{div}$ and that
$\varphi_0\in H$ such that $F(\varphi_0)\in L^1(\Omega)$.
For every $k\in\mathbb{N}$ let us define $\varphi_{0k}\in D(B)$ as
$$\varphi_{0k}:=\Big(I+\frac{1}{k}B\Big)^{-1}\varphi_0.
$$
Since $B$ is maximal and monotone, we have $\varphi_{0k}\to\varphi_0$ in $H$.
Let $[u_k,\varphi_k]$ be a weak solution corresponding to $u_0$ and $\varphi_{0k}$
and satisfying \eqref{df1}-\eqref{ic}.
We know that $[u_k,\varphi_k]$ satisfies the energy inequality \eqref{ei} for each $k$,
on the right hand side of which we need to control the nonlinear term that,
by virtue  \eqref{convex}, can be written
as
\begin{equation}
\int_{\Omega}F(\varphi_{0k})=\int_{\Omega}G(\varphi_{0k})-\frac{a^{\ast}}{2}\|\varphi_{0k}\|^2.\label{nl}
\end{equation}
To this aim we multipy the equation $\varphi_{0k}-\varphi_0=-\frac{1}{k}B\varphi_{0k}$
by $g(\varphi_{0k})$ in $L^2(\Omega)$, where $g=G'$. We obtain
\begin{eqnarray}
& &\int_{\Omega}g(\varphi_{0k})(\varphi_{0k}-\varphi_0)=-\frac{1}{k}\int_{\Omega}g(\varphi_{0k})B\varphi_{0k}
\nonumber\\
& &=-\frac{1}{k}\int_{\Omega}g'(\varphi_{0k})|\nabla\varphi_{0k}|^2-\frac{1}{k}\int_{\Omega}g(\varphi_{0k})
\varphi_{0k}\leq 0,
\end{eqnarray}
since $g$ is monotone nondecreasing and we can suppose that $g(0)=0$.
Therefore, due to the convexity of $G$ we can write
\begin{eqnarray}
& & \int_{\Omega}G(\varphi_{0k})\leq\int_{\Omega}G(\varphi_0)+\int_{\Omega}g(\varphi_{0k})
(\varphi_{0k}-\varphi_0)\leq\int_{\Omega}G(\varphi_0).\label{u}
\end{eqnarray}
Hence, on account of \eqref{nl} and \eqref{u}  we get the desired
control and from \eqref{ei}, written for each weak solution
$[u_k,\varphi_k]$, by means of (H3) and of Gronwall lemma, we
deduce the estimates \eqref{e1}, \eqref{e2} and \eqref{e3} for
$u_k$, $\varphi_k$ and $\nabla\mu_k$, respectively. By taking the
gradient of
$\mu_k=a\varphi_k-J\ast\varphi_k+F^\prime(\varphi_k)$,
multiplying the resulting relation by $\nabla\varphi_k$ in
$L^2(\Omega)$ and using (H2) we recover the control of the gradient
of $\varphi_k$ from the gradient of $\mu_k$ (see \eqref{y}) and
therefore, for $\varphi_k$ we get the estimate \eqref{e5}. Moreover,
arguing as in the Faedo-Galerkin approximation scheme above we get
\eqref{e6} and \eqref{wg1} for $\mu_k$ and
$\rho(\cdot,\varphi_k)$, and \eqref{e7}, \eqref{wg2} and
\eqref{dertphi1}-\eqref{dertphi3} for the time derivatives $u_k'$
and $\varphi_k'$, respectively. By compactness we hence deduce the
existence of four functions $u$, $\varphi$, $\mu$ and $\rho$
satisfying \eqref{p1}-\eqref{p2} such that the convergences
\eqref{c1}-\eqref{c10} hold. By passing to the limit in the variational
formulation for $[u_k,\varphi_k]$ it is immediate to see that
$[u,\varphi]$ is a solution corresponding to the initial data $u_0$ and
$\varphi_0$. This completes the proof of the existence of a weak
solution when $u_0\in G_{div}$ and $\varphi_0\in H$ such that
$F(\varphi_0)\in L^1(\Omega)$.

Finally, the energy inequality \eqref{ei} for the solution $[u,\varphi]$
can be obtained by passing to the limit in the energy inequality
\eqref{ei} written for each approximating couple $[u_k,\varphi_k]$,
using the weak/strong convergences \eqref{c1}-\eqref{c10} and
Fatou's lemma, in a similar way as done above for the Faedo-Galerkin
approximate solutions (see \eqref{wq1}-\eqref{wq11}). In
particular, on account of \eqref{convex}, when we pass to the limit in
the nonlinear term on the right hand side we have, by \eqref{nl},
\eqref{u} we infer
$$\limsup_{k\to\infty}\int_{\Omega}F(\varphi_{0k})\leq\int_{\Omega}G(\varphi_0)-\frac{a^{\ast}}{2}\|\varphi_0\|^2
=\int_{\Omega}F(\varphi_0).$$
The proof of Theorem \ref{thm} is now complete.

\begin{oss}
\label{smooth} {\upshape If we compare estimates \eqref{e7} and
\eqref{wg2} for the time derivatives $u_n'$ in the case $d=3$ and
$d=2$, respectively, with the analogous estimates that hold in the
case of the local Cahn-Hilliard-Navier-Stokes system (see, e.g.,
\cite{B}), we see that in the case $d=3$ we obtain the same time
regularity exponent $4/3$ for both the local and nonlocal systems.
However, in the local system we can estimate $\varphi_n$ in
$L^{\infty}(0,T;V)$ so that, in two dimensions we easily get the
exponent $2$. For the nonlocal system, this possibility seems out of
reach since we can only estimate $\varphi_n$ in $L^2(0,T;V)$. Also,
for the same reason, the transport term in the Cahn-Hilliard equation
is less regular so that the bound on $\varphi^\prime_n$ is weaker in
comparison with the analog for the local system.}
\end{oss}

\begin{oss}
{\upshape We point out that energy inequality \eqref{ei} can be
written in an alternative form, provided that a suitable condition holds.
Indeed, suppose that
\begin{equation}
\int_{\Omega}{\varphi_0}=0,
\end{equation}
and that $c_0$ (see (H2)) and $J$ are such that
\begin{equation}
C_P<\frac{c_0}{2\|\nabla J\|_{L^1}},\label{altass}
\end{equation}
where $C_P$ is the Poincar\'{e}-Wirtinger constant in the inequality
$$\|\varphi\|\leq C_P\|\nabla\varphi\|,\qquad
\forall\varphi\in V\mbox{ s.t. }\int_{\Omega}\varphi=0.$$
Then, we can get the following control of the gradient of $\varphi$
by the gradient of $\mu$
\begin{equation}
\|\nabla\mu\|^2\geq\beta\|\nabla\varphi\|^2,\label{nei}
\end{equation}
where $\beta=(c_0-2C_P\|\nabla J\|_{L^1})^2$ (compare \eqref{nei} with \eqref{y}).
Indeed, by taking the gradient of $\mu=a\varphi-J\ast\varphi+F'(\varphi)$,
multiplying the resulting relation by $\nabla\varphi$ and using (H2) we have
\begin{eqnarray}
& &\frac{\sqrt{\beta}}{2}\|\nabla\varphi\|^2+\frac{1}{2\sqrt{\beta}}\|\nabla\mu\|^2\geq (\nabla\mu,\nabla\varphi)\nonumber\\
& &\geq c_0\|\nabla\varphi\|^2-2\|\nabla J\|_{L^1}\|\varphi\|\|\nabla\varphi\|\nonumber\\
& &\geq (c_0-2C_P\|\nabla J\|_{L^1})\|\nabla\varphi\|^2=\sqrt{\beta}\|\nabla\varphi\|^2,
\end{eqnarray}
whence \eqref{nei}. Therefore, as a consequence of \eqref{ei},
for the weak solution $[u,\varphi]$ of Theorem \ref{thm} the following
energy inequality is satisfied as well
\begin{eqnarray}
& &\frac{1}{2}\Big(\|u(t)\|^2+
\frac{1}{2}\int_{\Omega}\int_{\Omega}J(x-y)(\varphi(x,t)-\varphi(y,t))^2 dxdy+2\int_{\Omega}F(\varphi(t))
\Big)\nonumber\\
& &+\int_0^t(\nu\|\nabla u(\tau)\|^2+\beta\|\nabla\varphi(\tau)\|^2)d\tau\nonumber\\
& &\leq
\frac{1}{2}\Big(\|u_0\|^2+\frac{1}{2}\int_{\Omega}\int_{\Omega}J(x-y)(\varphi_0(x)-\varphi_0(y))^2 dxdy+2\int_{\Omega}F(\varphi_0)
\Big)\nonumber\\
& &+\int_0^t\langle h(\tau),u(\tau)\rangle d\tau.
\end{eqnarray}
We recall that $C_P$ can be estimated for many important special
classes of domains (cf., e.g., \cite{L}). For example, if $\Omega$ is
convex we can take $C_P=\mbox{diam}(\Omega)/\pi$ and there
exist convex domains for which this constant is optimal (see
\cite{Be})}.
\label{rem9}
\end{oss}


\section{Proofs of Corollaries \ref{cor} and \ref{cor2}}\setcounter{equation}{0}

\begin{proof}[Proof of Corollary \ref{cor}]
Recalling \eqref{coerc} and repeating the proof of Theorem \ref{thm}, in place of \eqref{w2}
we have
\begin{eqnarray}
& &\frac{1}{2}\int_{\Omega}\int_{\Omega}J(x-y)(\varphi_n(x)-\varphi_n(y))^2 dxdy
+2\int_{\Omega}F(\varphi_n)\nonumber\\
& &=\|\sqrt{a}\varphi_n\|^2+2\int_{\Omega}F(\varphi_n)-(\varphi_n,J\ast\varphi_n)\nonumber\\
& &\geq\int_{\Omega}((a-\|J\|_{L^1})\varphi_n^2+2c_7|\varphi_n|^{2+2q})-2c_8|\Omega|
\geq c\|\varphi_n\|_{L^{2+2q}(\Omega)}^{2+2q}-c,
\end{eqnarray}
and this estimate, by integrating \eqref{est1} as done above, allows to control
the sequence of $\varphi_n$ and yields \eqref{impr0}.
All the other estimates for $\varphi_n$, $u_n$, $\mu_n$ and $\rho(\cdot,\varphi_n)$
established in the proof of Theorem \ref{thm} still hold. The only estimates that can be improved
are the ones for $u_n'$ and $\varphi_n'$. Indeed,
for $d=2$, in place of \eqref{td33} we can write
$$\|\widetilde{P}_n(\varphi_n\nabla\mu_n)\|_{V_{div}'}\leq c\|\varphi_n\|_{L^{2+2q}(\Omega)}\|\nabla\mu_n\|
\leq N\|\nabla\mu_n\|,$$
and hence we can control the sequence of $\widetilde{P}_n(\varphi_n\nabla\mu_n)$ in $L^2(0,T;V_{div}')$.
This control, combined with the control for the other terms in \eqref{td1},
yields \eqref{u_tnew}.
Furthermore, as far as the sequence of $\varphi_n'$ is concerned,
we can improve estimates \eqref{derphi0}-\eqref{derphi2} by arguing as in the proof
of Theorem \ref{thm} and by considering the following cases.
Choosing $s=((4-d)p+2d)/2p$ (cf. \eqref{Vs} and \eqref{derphi1}),
when $1<p\leq d'=d/(d-1)$, due to the embeddings $H^{s-1}(\Omega)\hookrightarrow L^{\infty}(\Omega)$
(if $1<p<d'$) or $H^{s-1}(\Omega)\hookrightarrow L^r(\Omega)$ for every $r<\infty$ (if $p=d'$),
 we have
\begin{eqnarray}
& &\Big|\int_{\Omega}(u_n\cdot\nabla\psi_I)\varphi_n\Big|\leq c\|u_n\|\|\varphi_n\|_{L^{2+2q}(\Omega)}\|\psi\|_{V_s}\leq N\|\psi\|_{V_s}.\label{mod1}
\end{eqnarray}
The same estimate also holds for the case $d=3$ when $3/2<p\leq 2$ and $q\geq 2(2p-3)/(6-p)$,
where here we use the embedding $H^{s-1}(\Omega)\hookrightarrow L^{3p/(2p-3)}(\Omega)$ and the fact that $6p/(6-p)\leq 2+2q$.
Finally, when $d=3$, $3/2<p\leq 2$ and $0<q< 2(2p-3)/(6-p)$ we have
\begin{eqnarray}
& &\Big|\int_{\Omega}(u_n\cdot\nabla\psi_I)\varphi_n\Big|\leq
c\|u_n\|\|\varphi_n\|_{L^{6p/(6-p)}(\Omega)}\|\psi\|_{V_s}\nonumber\\
& &\leq c\|u_n\|
\|\varphi_n\|_{L^{2+2q}(\Omega)}^{2(3-p)(1+q)/p(2-q)}\|\varphi_n\|_{L^6(\Omega)}^{(4p-6q+pq-6)/p(2-q)}
\|\psi\|_{V_s}\nonumber\\
& &\leq N\|\varphi_n\|_V^{(4p-6q+pq-6)/p(2-q)}\label{mod2}
\|\psi\|_{V_s}.
\end{eqnarray}
Hence, on account of \eqref{mod1} and \eqref{mod2}, from \eqref{ap1} (written with $\psi=\psi_{I}$)
we deduce \eqref{impr1} and \eqref{beta}. The improved regularity \eqref{impr2} for $\varphi_t$
can be obtained by estimating the term $(u,\varphi\nabla\psi)$ in \eqref{phivar} for the case $d=2$ as
\begin{equation}
|(u,\varphi\nabla\psi)|\leq c\|\nabla u\|\|\varphi\|_{L^{2+2q}(\Omega)}\|\nabla\psi\|
\leq N\|\nabla u\|\|\nabla\psi\|,\nonumber
\end{equation}
and for the case $d=3$ and $q\geq 1/2$ as
\begin{eqnarray}
& &|(u,\varphi\nabla\psi)|\leq c \|u\|_{L^{2(1+1/q)}(\Omega)}\|\varphi\|_{L^{2+2q}(\Omega)}\|\nabla\psi\|
\leq N\|\nabla u\|\|\nabla\psi\|\nonumber.
\end{eqnarray}
\end{proof}

\begin{proof}[Proof of Corollary \ref{cor2}]
For $d=2$ the regularity properties \eqref{impr2} and \eqref{u_tnew} allow us to deduce
the energy identity for the weak solution. Indeed, in this case we can take and $v=u(\tau)$ in \eqref{wf2} and $\psi=\mu(\tau)$ in \eqref{phivar}, sum the resulting equations and then integrate
with respect to $\tau$ between $0$ and $t$. When we consider the duality product $\langle\varphi_t,\mu\rangle$, we are led
to the duality $\langle\varphi_t,F'(\varphi)\rangle$ which can be rewritten by taking
into account that $F'(\varphi)=g(\varphi)-a^{\ast}\varphi$, with $g\in C^1(\mathbb{R})$ monotone
increasing. Now, introducing the functional $\mathcal{G}:H\to\mathbb{R}\cup\{+\infty\}$
defined as $\mathcal{G}(\varphi)=\int_{\Omega}G(\varphi)$ if $G(\varphi)\in L^1(\Omega)$ and
$\mathcal{G}(\varphi)=+\infty$ otherwise, we have (see \cite[Proposition 2.8, Chap. II]{BAR})
that $\mathcal{G}$ is convex, lower semicontinous on $H$ and $\xi\in\partial\mathcal{G}(\varphi)$
if and only if $\xi=G'(\varphi)=g(\varphi)$ almost everywhere in $\Omega$. In view of \eqref{impr2} and of the fact
that $g(\varphi)\in L^2(0,T;V)$, we can use \cite[Proposition 4.2]{CKRS} and get, for almost any $t\in(0,T)$
\begin{eqnarray}
& &\langle\varphi_t,F'(\varphi)\rangle
=\langle\varphi_t,g(\varphi)\rangle-a^{\ast}\langle\varphi_t,\varphi\rangle
=\frac{d}{dt}\Big(\mathcal{G}(\varphi)-\frac{a^{\ast}}{2}\|\varphi\|^2\Big)
=\frac{d}{dt}\int_{\Omega}F(\varphi).\nonumber\label{chain}
\end{eqnarray}
Therefore, on account of this identity, from \eqref{wf2} and \eqref{phivar} we obtain
\begin{eqnarray}
& &\frac{1}{2}\frac{d}{dt}\Big(\|u\|^2+\|\sqrt{a}\varphi\|^2-(\varphi,J\ast\varphi)
+2\int_{\Omega}F(\varphi)\Big)+\nu\|\nabla u\|^2+\|\nabla\mu\|^2\nonumber\\
& &=\frac{1}{2}\frac{d}{dt}\Big(\|u\|^2+\frac{1}{2}\int_{\Omega}\int_{\Omega}J(x-y)(\varphi(x)-\varphi(y))^2 dxdy
+2\int_{\Omega}F(\varphi)\Big)\nonumber\\
& &+\nu\|\nabla u\|^2+\|\nabla\mu\|^2=\langle h,u\rangle,
\end{eqnarray}
Hence we get \eqref{idendiffcor}. Furthermore,
by integrating between $0$ and $t$ we get the energy identity in integral form, i.e, \eqref{ei}
holds with the equal sign for every $t\geq 0$.

In order to obtain \eqref{dissest},
let us multiply equation $\mu=a\varphi-J\ast\varphi+F'(\varphi)$ by $\varphi$ in $L^2(\Omega)$.
We obtain
\begin{equation}
(\mu,\varphi)=\frac{1}{2}
\int_{\Omega}\int_{\Omega}J(x-y)(\varphi(x)-\varphi(y))^2 dxdy+(F'(\varphi),\varphi).\label{muphi}
\end{equation}
Now, observe that, due to (3.18) and to the convexity of $G$ we have
$$F(0)\geq F(s)+\frac{a^{\ast}}{2}s^2-(F'(s)+a^{\ast}s)s$$
and hence
$$F'(s)s\geq F(s)-\frac{a^{\ast}}{2}s^2-F(0).$$
Therefore, from \eqref{muphi} we get
\begin{eqnarray}
& &(\mu,\varphi)\geq\frac{1}{2}
\int_{\Omega}\int_{\Omega}J(x-y)(\varphi(x)-\varphi(y))^2 dxdy+\int_{\Omega}F(\varphi(t))\nonumber\\
& &-\frac{a^{\ast}}{2}\|\varphi\|^2-c.\label{muphi2}
\end{eqnarray}
Setting $\overline{\mu}=\frac{1}{|\Omega|}\int_{\Omega}\mu$ and suppose $(\varphi_0,1)=0$
for simplicity. Then we have
$$(\mu,\varphi)=(\mu-\overline{\mu},\varphi)\leq C_p\|\nabla\mu\|\|\varphi\|,$$
and then, by means of (H6), from \eqref{muphi2} we have
\begin{eqnarray}
& &\frac{1}{8}\int_{\Omega}\int_{\Omega}J(x-y)(\varphi(x)-\varphi(y))^2 dxdy+\frac{1}{2}\int_{\Omega}F(\varphi)
+\frac{c_7}{2}\int_{\Omega}|\varphi|^{2+2q}-c_9\nonumber\\
& &-\frac{a^{\ast}}{2}\|\varphi\|^2-c\leq\frac{3}{2}\|J\|_{L^1}\|\varphi\|^2+\|\nabla\mu\|^2+\frac{C_P^2}{2}
\|\varphi\|^2.\nonumber
\end{eqnarray}
Therefore, we deduce
\begin{equation}
\frac{1}{8}\int_{\Omega}\int_{\Omega}J(x-y)(\varphi(x)-\varphi(y))^2 dxdy+\frac{1}{2}\int_{\Omega}F(\varphi)
\leq\|\nabla\mu\|^2+c_{10}\nonumber
\end{equation}
and hence
\begin{equation}
\frac{1}{2}\mathcal{E}(u,\varphi)\leq c_{11}\Big(\frac{\nu}{2}\|\nabla u\|^2+\|\nabla\mu\|^2\Big)+c_{10},\label{hr}
\end{equation}
where $c_{11}=\max(1,1/2\lambda_1\nu)$, $\lambda_1$ being the lowest eigenvalue of the Stokes operator $A$.
We point out that all constants only depend on the parameters
of the problem and are independent of the initial data.
Now, by virtue of \eqref{idendiffcor} and \eqref{hr} we have
\begin{equation}
\frac{d}{dt}\mathcal{E}(u,\varphi) + k\mathcal{E}(u,\varphi)\leq l+\frac{1}{2\nu}\|h\|_{V_{div}'}^2,
\end{equation}
where $k=1/2c_{11}$ and $l=c_{10}/c_{11}$.
By means of Gronwall lemma we hence deduce
\begin{equation}
\label{dissest-0}
\mathcal{E}(u(t),\varphi(t))\leq \mathcal{E}(u_0,\varphi_0)e^{-kt}+K.
\end{equation}
with
$$K=\frac{l}{k}+\frac{1}{2\nu}\|h\|^2_{L^2(0,\infty;V_{div}')}.$$
If $m:=(\varphi_0,1)\not=0$, observe that if $[u,\varphi]$ is a weak solution
with data $[u_0,\varphi_0]$ for the problem with potential $F$, then
$[u,\widetilde{\varphi}]$, where $\widetilde{\varphi}=\varphi-m$
is a weak solution with data $[u_0,\varphi_0-m]$ for the same problem with
potential $\widetilde{F}$ given by
$$\widetilde{F}(s):=F(s+m)-F(m).$$
By relying on \eqref{dissest-0} satisfied by the solution $[u,\widetilde{\varphi}]$,
we easily get \eqref{dissest}.
\end{proof}


\medskip

\noindent {\bf Acknowledgments.} Some financial support from the
Italian MIUR-PRIN Research Project 2008 ``Transizioni di fase,
isteresi e scale multiple'' and from the IMATI of CNR~in~Pavia,~Italy,
is gratefully acknowledged. The second author was also supported by
the FTP7-IDEAS-ERC-StG Grant $\sharp$200497(BioSMA).

\medskip\noindent


\end{document}